\documentclass[11pt]{article}

\usepackage{amssymb,amsmath,euscript}
\usepackage{amsfonts}
\usepackage[authoryear]{natbib}

\usepackage{verbatim}

\def\@enum@{\list{\csname label\@enumctr\endcsname}%
           {\usecounter{\@enumctr}\def\makelabel##1{\hss\llap{##1}}
           \itemsep=2pt\parsep=0pt\topsep=3pt plus 1pt minus 1 pt}}
\usepackage{paralist}
\usepackage{soul,color}

\newtheorem{proposition}{Proposition}[section]
\newtheorem{lemma}[proposition]{Lemma}
\newtheorem{theorem}[proposition]{Theorem}

\newtheorem{cor}[proposition]{Corollary}

\renewcommand{\d}[1]{|| #1 ||_\infty}

\def\a{\alpha}

\def\de{\delta}
\def\ep{\varepsilon}

\def\o{\omega}
\def\q{\quad}
\def\dim{{\rm dim}_{_{\rm H}}}
\def\dimp{{\rm dim}_{_{\rm P}}}

\def\vep{\varepsilon}

  \def\sF {{\cal F}}

\def\ESC {\EuScript C}

\def\bP {{\mathbb P}}

\def\I{{\rm 1 \hskip-2.9truept l}}

\newcommand{\wh}{\widehat}

\def\R{{\mathbb R}}

\def\PP{{\mathbb P}}
\def\E{{\mathrm E}}
\def\P{{\mathrm P}}
\def\Z{{\mathbb Z}}

\makeatletter \@addtoreset{equation}{section} \makeatother
\renewcommand{\theequation}{\arabic{section}.\arabic{equation}}
\newcommand {\qed}%
{%
    {}\hfill
    {}\hfill
    {$\square $}%
    \vspace {0.3cm}%
    \pagebreak [2]%
    \par
}%
\newenvironment{proof}[1]{%
    \vspace{0.3cm}%
    \pagebreak [2]%
    \par%
    \noindent {\bf  Proof~#1\ }}{\qed}%
\newenvironment{remark}{%
    \vspace{0.3cm} \pagebreak [2]%
    \par%
    \refstepcounter{proposition}
    \noindent%
    {\bf Remark~\theproposition\  }}{\qed}%
\begin{document}

\title{Discrete Fractal Dimensions of the Ranges of Random Walks in $\Z^d$
Associate with Random Conductances}
\author{Yimin Xiao\footnote{Research partially
        supported by NSF grant DMS-1006903}\\
Michigan State University\\
\and Xinghua Zheng \footnote{Research partially
        supported by GRF 606010 of the HKSAR.}
\\ Hong Kong University of Science and Technology\\}

\maketitle
\begin{abstract}
Let $X= \{X_t, t \ge 0\}$ be a continuous time random walk in
an environment of \hbox{i.i.d.} random conductances $\{\mu_e \in [1, \infty), e
\in E_d\}$, where $E_d$ is the set of nonoriented nearest neighbor bonds on
the Euclidean lattice $\mathbb{Z}^d$ and $d\ge 3$.
Let ${\mathrm R} = \{x \in \mathbb{Z}^d: X_t = x \hbox{ for some } t \ge 0\}$ be
the range of $X$. It is proved that, for almost every realization of the
environment, ${\rm dim}_{_{\rm H}} {\mathrm R} = {\rm dim}_{_{\rm P}} {\mathrm R} = 2$ almost surely,
where ${\rm dim}_{_{\rm H}}$ and ${\rm dim}_{_{\rm P}}$ denote respectively the discrete Hausdorff and packing
dimension. Furthermore, given
any
set $A \subseteq \mathbb{Z}^d$, a criterion for $A$
to be hit by $X_t$ for  arbitrarily large $t>0$ is given in terms of ${\rm dim}_{_{\rm H}} A$.
Similar results for Bouchoud's trap model in $\mathbb{Z}^d$ ($d \ge 3$)
are also proven.
\end{abstract}

{\sc Running head}:  Fractal Dimensions of the Range of the Random
Conductance Model\\

{\sc 2010 AMS Classification numbers}: 60K37, 60F17, 82C41, 31C20.\\


\section{Introduction}
\label{sec:Intro}

Ordinary fractal dimensions such as Hausdorff dimension and packing dimension
are useful tools not only for analyzing the (microscopic) geometric structures
of various thin sets and measures in the Euclidean space $\R^d$, but also for
many scientific applications; see Falconer \cite{Fal90} for a systematic account.
In probability theory, they have been applied to study fine properties of the
sample paths of Brownian motion, L\'evy processes and random fields. We refer
to  Taylor \cite{Taylor86},  Xiao \cite{X04, X09} for more information. Many
discrete sets, such as percolation clusters, also exhibit (macroscopic or
global) fractal phenomena. In order to investigate their geometric structures,
Barlow and Taylor \cite{BT89,BT92} have introduced the notions of discrete
Hausdorff and packing dimensions and used them to study the fractal properties of
strictly stable random walks. See also Khoshnevisan \cite{Kh94}.

In this paper, we apply discrete Hausdorff and packing dimensions of
Barlow and Taylor \cite{BT92} to describe the range of a class of random walks in random
environment, namely the random conductance models (RCM) on
the Euclidean lattice $\Z^d$ considered by Barlow and Deuschel \cite{BD10},
among others.

More specifically, for $x,y\in \Z^d$, we say that $x\sim y$ if $x$ and $y$ are neighboring sites,
(i.e., $|x-y|=1$, where $|\cdot|$ is the Euclidean distance) and $x\not\sim y$ otherwise.
Let $E_d$ be the set of nonoriented nearest neighbor bonds,
i.e., $E_d  =\{e=(x,y):\ x\sim y\}$, and
let  $\{\mu_e, e \in E_d\}$ be a sequence of
non-negative \hbox{i.i.d.} random variables with values in $[1, \infty)$, defined on a
probability space $(\Omega, \PP)$. We may take $\Omega = [1, \infty)^{E_d}$, the
set of configurations of conductances, and let $\PP$ be the product probability
measure on $\Omega$ under which the coordinates $\mu_e, e \in E_d$, are
\hbox{i.i.d.} random variables.

We write $\mu_{xy} =\mu_{(x,y)} = \mu_{yx}$, let $\mu_{xy} = 0$ if
$x \nsim y$ and set $\mu_x = \sum_y \mu_{xy}$. There are two natural continuous
time random walks on $\Z^d$ associated with $\{\mu_e, e \in E_d\}$.
Both jump from $x$ to $y \sim x$ with probability $P(x, y) = \mu_{xy}/\mu_x$.
The first (the variable speed random walk or ${\mathrm {VSRW}}$) waits at
$x$ for an exponential time with mean $1/\mu_x$ while the second (the constant
speed random walk or ${\mathrm {CSRW}}$) waits at $x$ for an exponential time
with mean 1. Their generators ${\cal L}_{V}$ and ${\cal L}_{C}$ are given by
\begin{equation}\label{Def:generator}
\begin{split}
   {\cal L}_{V}(\o)f(x) &= \sum_y\mu_{xy}(\o)\big(f(x) - f(y)\big),\q \mbox{and}\\
   {\cal L}_{C}(\o)f(x) &= \mu_x(\o)^{-1} \sum_y\mu_{xy}(\o)\big(f(x) - f(y)\big),
   \end{split}
\end{equation}
respectively. VSRW is reversible with
stationary measure
$\nu$ defined by
$\nu(\{x\}) = 1$, $x \in \Z^d$; and CSRW is reversible with $\mu_x, x \in \Z^d$
as its
stationary
measure. Since the generators ${\cal L}_{V}$ and ${\cal L}_{C}$
only differ by a multiple, ${\mathrm {VSRW}}$ and ${\mathrm{CSRW}}$ are time
change of each other; see Barlow and Deuschel \cite[pp.39-40]{BD10}
for precise information.

The random conductance model has been studied by several authors
under various restrictions on the law of $\mu_e$. There are three
typical cases: $c^{-1}\le \mu_e \le c$ for some $c\ge 1$ (strong ellipticity),
$0\le \mu_e \le 1$, and $1 \le \mu_e < \infty$. An important example of
the random conductance models is a continuous time simple
random walk on a supercritical percolation cluster ${\cal C}_\infty$ in~$\Z^d$.
In this case $\{\mu_e, e \in E_d\}$ are \hbox{i.i.d.} Bernoulli random variables with
mean $p>p_c(d)$, the critical probability for bond percolation in~$\Z^d$. See
Barlow \cite{B04}, Berger et al. \cite{BeB08}, Biskup and Prescott \cite{BP07},
Mathieu \cite{M08} and the references therein for further information.

Under the assumptions that $d \ge 2$ and $\PP(\mu_e \ge 1)= 1$, Barlow and Deuschel
\cite{BD10} prove that the VSRW satisfies a quenched functional central
limit theorem and the limiting process is $\sigma_V B$,
where $\sigma_V>0$ is a nonrandom constant and $B$ is a Brownian motion
on $\R^d$. As shown by Barlow and \u Cern\' y \cite{BC11},
Barlow and Zheng \cite{BZ10} and  \u Cern\' y  \cite{C11}, the scaling limit of
${\mathrm {CSRW}}$ in $\Z^d$
with $d \ge 2$ in the heavy-tailed environment can either be a Brownian motion or a
fractional-kinetics process (which is a Brownian motion time-changed by the inverse of
a stable subordinator). Random conductance models under the general conditions
$\PP\big(0 \le \mu_e < \infty \big) =1$ and $\PP\big(\mu_e >0 \big) > p_c(d)$
have been recently investigated by Andres et al. \cite{ABDH10}.
In this paper
we focus on the case $\PP(\mu_e \ge 1)= 1$.

This paper is concerned with fractal properties of the ranges of random
conductance models. Since the time change which relates CSRW and VSRW is
strictly increasing and continuous,
${\mathrm {VSRW}}$ and ${\mathrm {CSRW}}$ have the
same range.
Hence, in the following, we consider ${\mathrm {VSRW}}$ in $\Z^d$
and denote it by $X$.
Also, for any environment $\{\mu_e(\o), e \in E_d\}$ and any $x\in \Z^d$, we write $\P^{\o}_x$
for the (quenched) law of $X$ started at $x$.

Let
\[
{\mathrm R} = \{x \in \Z^d: \, X_t = x \hbox{ for some } t \ge 0\}
\]
be the range of VSRW $X$ in $\Z^d$. It follows from Theorem 1.2 of Barlow and Deuschel
\cite{BD10} that when $d\geq 2$ $X$ is transient if and only if $d \ge 3$.
Hence for $d =2$, $X$ is recurrent and R~=~$\Z^2$ $\P^\o_0$-a.s.
The case of $d=1$ is similar because
by, for example, Lemma 1.5 of Solomon \cite{Solomon75}, the range is almost
surely the whole line. We shall henceforth assume that $d\geq 3$.

The following are our main theorems, which describe the fractal structures of R
and characterize the transient sets for $X$ by using the discrete Hausdorff and
packing dimensions defined by Barlow and Taylor \cite{BT89, BT92}.

\begin{theorem}\label{Th:main}
Assume that $d \ge 3$ and $\PP(\mu_e \ge 1)= 1$. Then for $\PP$-almost every
$\o \in \Omega$,
\[
\dim {\mathrm R} = \dimp {\mathrm R} = 2, \quad \P^\o_0\hbox{-a.s.},
\]
where $\dim$ and $\dimp$ denote respectively the discrete Hausdorff and packing
dimension.
\end{theorem}

\begin{theorem}\label{Th:main2}
Assume that $d \ge 3$ and $\PP(\mu_e \ge 1)= 1$. Let $A \subset \Z^d$
be any (infinite)
set.
Then for $\PP$-almost every $\o \in \Omega$, the following statements hold.
\begin{itemize}
\item[(i)]\, If $\dim A < d-2$, then
\begin{equation*}\label{Eq:wiener1}
\P^\o_0\big(X_t \in A\ \hbox{ for  arbitrarily large } t >0\big) = 0.
\end{equation*}
\item[(ii)]\, If $\dim A > d-2$, then
\begin{equation*}\label{Eq:wiener2}
\P^\o_0\big(X_t \in A\ \hbox{ for  arbitrarily large } t >0\big)
=1.
\end{equation*}
\end{itemize}
\end{theorem}

The above theorems show that, if $\mu_e \ge 1$, then for almost every realization of the
environment, VSRW and CSRW have long term fractal and asymptotic behavior
similar to the simple random walk on $\Z^d$ and Brownian motion in $\R^d$.


\begin{remark}\label{rmk:general_cases}
\begin{compactenum}[(i)]
\item When  $c^{-1} \le \mu_e \le c$ for some constant $c\ge 1$,
Barlow and Deuschel \cite[Remark 6.3 and Theorem 6.1]{BD10} prove that,
if $\{\mu_e, e \in E_d\}$ is stationary, symmetric and ergodic, then
Lemmas~\ref{Lem:BD} and \ref{01_law} below still hold.
(See also Delmott \cite{D99} for \hbox{i.i.d.} environment.)
As the proofs of Theorems~\ref{Th:main} and~\ref{Th:main2} only make use of
Lemmas \ref{Lem:BD} and \ref{01_law},
it follows that
in this case the
independence assumption on $\{\mu_e\}$ in Theorems \ref{Th:main} and \ref{Th:main2}
can be weakened and the same conclusions remain valid.

\item When $\PP\big(0 \le \mu_e \le 1\big)= 1$, only partial estimates on the
heat kernel of $X$ on the diagonal are available, see \cite{BeB08,BP07,M08}.
Berger et al. \cite{BeB08} show that Gaussian heat kernel bounds do not hold
in general. This is caused by traps due to edges in $E_d$ with small positive
conductances.

Under the extra condition that $\PP\big(\mu_e >0\big)> p_c(d)$, (otherwise the
range R of $X$ is a finite set,) Andres et al. \cite{ABDH10} prove that
the Green's function of $X$ satisfies bounds in (d) of Lemma \ref{Lem:BD} below,
but their Remark 7.6 shows that (e) of Lemma \ref{Lem:BD} does not hold in general.
Since our proofs of Theorems \ref{Th:main} and \ref{Th:main2} rely heavily on
Lemma \ref{Lem:BD}, it is not known whether similar results still hold. We will
consider these and related problems separately.
\end{compactenum}
\end{remark}

The proofs of Theorems \ref{Th:main} and \ref{Th:main2} are similar to those of
Theorems 7.8 and 8.3 of Barlow and Taylor \cite{BT92},
where transient, strictly $\alpha$-stable random walks on $\Z^d$
are treated. However, there are significant differences between VSRW and strictly
stable random walks. One major difference is that VSRW is not a random walk in
the classical sense since it does not have \hbox{i.i.d.} increments.
We make use of general Markov techniques to derive hitting
probability estimates and maximal inequality for VSRW, and also to overcome the
difficulties caused by the dependence between the increments, see, e.g., Lemma
\ref{Lem:SLLN}. These results and the proofs of Theorems \ref{Th:main} and \ref{Th:main2} are
given in Section 3. Since our arguments are based on general Markovian techniques,
they will be useful for studying other properties of VSRW, as well as more
general Markov chains.

We also consider another kind of random walk in random environment,
namely Bouchaud's trap model (BTM). This model was first introduced in the physics
literature to explain some strange dynamical properties of complex disordered systems,
in particular aging. We refer to Barlow and \u Cern\' y \cite{BC11} for a brief
historical account on BTM and to Ben Arous and \u Cern\' y \cite{BenC}, Barlow and \u Cern\' y
\cite{BC11} and \u Cern\' y \cite{C11} for results on scaling limits.

To recall the definition of BTM, let $\{\kappa_x, x \in \Z^d\}$ be \hbox{i.i.d.}
positive random variables on a probability space $(\tilde{\Omega}, \tilde{\PP})$.
For a given constant $a \in [0, 1]$, define random conductances $\tilde{\mu}_e$
($e\in E_d$ ) on $\Z^d$ by
$$
\tilde{\mu}_{xy} = \kappa_x^a\kappa_y^a,\qquad \hbox{ if }\ x \sim y.
$$
Then BTM is the continuous time Markov chain on $\Z^d$
whose transition rates $w_{xy}$ are given by
\[
w_{xy} = \frac{\tilde{\mu}_{xy}} {\kappa_x} = \kappa_x^{a-1}\kappa_y^a,
\qquad \hbox{ if }\ x \sim y.
\]
If $a=0$, then $\tilde{\mu}_e=1$ for all $e\in E_d$ and the BTM  is
a time change of the simple random walk on~$\Z^d$. If $ a\ne 0$, then, following
Barlow and \u Cern\' y \cite{BC11}, it is referred to as the non-symmetric BTM.

Just in the same way as for the RCM, we can define the VSRW, denoted by $\tilde{X}$, associated with the conductances $\{\tilde\mu_e\}$. The BTM and
$\tilde{X}$ are again related to each other by a time change, see equation (2.3) in \cite{BC11}, and so in particular have the same range.
Similarly as before, for any environment $\{\tilde\mu_e(\o), e \in E_d\}$ and any $x\in \Z^d$, we write $\tilde{\P}^{\o}_x$
for the (quenched) law of $\tilde{X}$ started at $x$.

Even though the conductances $\{\tilde{\mu}_{e}, e \in E_d\}$ are not
independent any more, they form a stationary symmetric ergodic process. By
applying the results in Barlow and Deuschel \cite{BD10} and Barlow and
\u Cern\' y \cite{BC11} to the VSRW $\tilde{X}$, we can use the method in Section 3
to prove the following theorem.

\begin{theorem}\label{Th:BTM}
Assume that $d \ge 3$ and $\tilde{\PP}\big(\kappa_x \ge 1\big)= 1$. Let $\tilde{\rm R}$
be the range of the Bouchaud's trap model. Then for $\tilde{\PP}$-almost every
$\o \in \tilde{\Omega}$,
\[
\dim \tilde{{\mathrm R}} = \dimp \tilde{{\mathrm R}} = 2, \quad
\tilde{\P}^\o_0\hbox{-a.s.}
\]
Moreover, the conclusions of Theorem \ref{Th:main2} hold for $\tilde X$.
\end{theorem}

The proof of Theorem \ref{Th:BTM} is given in Section 4. Similarly to
Remark~\ref{rmk:general_cases}~(ii), it would be interesting to determine
whether the assumption that $\kappa_x$ is bounded from below can be removed.

As is mentioned by an anonymous referee, in light of the above results
it would be interesting to investigate the discrete fractal dimensions of
percolation clusters and the images $X(E)$, where $E\subset \R_+$ and $X$
is VSRW or BTM, or an ordinary $\alpha$-stable random walk as in Barlow
and Taylor \cite{BT92}. We thank him/her for his/her thoughtful suggestions,
and we will study these questions in subsequent work.

Throughout this paper, for any $x,y\in\Z^d$, $|x-y|$ stands for the Euclidean distance, and
the
$\ell_\infty$ distance is denoted by
$\d{x-y}= \max_{i=1}^d |x_i - y_i|$.
We will use $c, c', c''$ etc to denote unspecified positive
and finite (nonrandom) constants, which may depend on the distribution of the
environment and may be different
in each occurrence. More specific constants are numbered as $c_{1}, \, c_{2},\cdots$.
\bigskip

\section{Preliminaries}

In this section, we recall some known facts about the ${\mathrm {VSRW}}$ and
discrete Hausdorff and packing dimensions, and prove a strong law of large numbers
(SLLN) for dependent events, which will be used in this paper.

\subsection{Some basic properties of VSRW}

Let $X= \{X_t, t \ge 0\}$ be a VSRW with values in $\Z^d$ with
$d \ge 3$ and let $p_t^{\o}(x, y) = \P^{\o}_x(X(t) = y)$ be its transition
density or the heat kernel of ${\cal L}_V$.

The following estimates for the transition density $p_t^{\o}(x, y)$
and the Green's function $g_t^{\o}(x, y)$ will be used in the sequel.
Recall that when $d \ge 3$, $X$ is transient and $g_t^{\o}(x, y)$ is
defined by
\[
g_t^{\o}(x, y) = \int_0^\infty p_t^{\o}(x, y)\, dt.
\]

\begin{lemma}\label{Lem:BD}
Let
$d \ge 3$,
$\PP(\mu_e \ge 1)= 1$, and $\eta \in (0, 1)$. There exist
random variables $\{U_x, x \in \Z^d\}$ and positive (nonrandom) constants
$c_i$ (depending on $d$ and the distribution of $\mu_e$) such that
\begin{equation*}\label{Eq:U}
\PP\big( U_x\ge n \big)\le c_1 \exp\big(- c_2 n^\eta\big).
\end{equation*}
\begin{compactenum}[(a)]
\item[(a)]\, {\rm \cite[Theorem 1.2(a)]{BD10}}
For all $x,y \in \Z^d$ and $t >0$,
\begin{equation*}
   p^\omega_{t}(x,y) \leq 1 \wedge \big(c_3 t^{-d/2} \big).
\end{equation*}

\item[(b)]\, {\rm \cite[Theorem 1.2(b)]{BD10}} If $|x-y|\vee\sqrt{t}\geq U_x$, then
\[
\aligned
p_t^\omega(x,y) \leq \left\{\aligned
& c_4 t^{-d/2}\exp(-c_5 |x-y|^2/t), & \mbox{  when } t\geq |x-y|,\\
& c_4\exp \bigr(-c_5 |x-y|(1 \vee \log(|x-y|/t) ) \bigl),
&\mbox{when } t\leq |x-y|.
     \endaligned\right.
\endaligned
\]

\item[(c)]\, {\rm \cite[Theorem 1.2(c)]{BD10}} If $t\geq
U_x^2\vee|x-y|^{1+\eta}$, then
\[
   p_t^\omega(x,y) \geq c_6 t^{-d/2}\exp(-c_7|x-y|^2/t).
\]

\item[(d)]\, {\rm \cite[Theorem 1.3]{BD10}} If $ |x-y|\ge U_x \wedge U_y$, then
\[
\frac{c_{8}} {|x-y|^{d-2}} \leq g^\omega(x,y)\leq \frac{ c_{9}}
{|x-y|^{d-2}}.
\]

\item[(e)]\, {\rm \cite[Lemma 3.4]{BC11}} 
For all $x,y \in \Z^d$,
\[
  g^\omega(x,y)\leq c_{10}.
\]

\item[(f)]\, {\rm \cite[Lemma 3.3]{BC11}}
There exists $c_{11}>0$ such that for each $K>0$, the inequality
\begin{equation}\label{eq:U_x}
   \max_{|x|\leq Kn} U_x\leq  c_{11} (\log n)^{1/\eta}
\end{equation}
holds with $\PP$-probability no less than $1-c_{12} K^d n^{-2}$.
In particular, $\PP$-\hbox{a.s.} there
exists $n_0 = n_0 (\o)$ such that \eqref{eq:U_x} holds for all $n \ge n_0$.
\end{compactenum}
\end{lemma}

In the rest of this paper, we take $\eta = 1/3$. Hence $\PP$-\hbox{a.s.} there
exists $n_0 = n_0 (\o)$ such that
\begin{equation} \label{Eq:n0}
\max_{\d{x}\le 2^n} U_x \le c_{11}\, n^3 \qquad \hbox{for all } \, n \ge n_0.
\end{equation}

We will sometimes work with the \emph{discrete time} VSRW $\wh X =
\{\wh X_n: n=0,1,\ldots\}$ defined by
$\wh X_n = X_n$ for $n=0,1,\ldots$. Its transition probabilities are nothing but
$p_n^\o(x,y)$, so in particular, satisfy (a), (b),~(c) of the
previous lemma. Define its Green's function as
\begin{equation}\label{def:green_discrete}
\wh{g}^\o(x,y)=\sum_{n=0}^\infty p_n^\o(x, y).
\end{equation}
Then using (a), (b) and (c) of Lemma \ref{Lem:BD} and similar
computations as in \S4.3 of Lawler and Limic \cite{LL10} one can derive
\begin{lemma}\label{Lem:green_discrete}
When $d\geq 3$, the inequalities in (d) and (e)  of Lemma
\ref{Lem:BD} also hold for $\wh{g}^\o(x,y)$.
\end{lemma}

We also recall the following connection between the hitting probabilities
of a time homogeneous transient Markov chain
$\{X_t,\, t \ge 0,\P_x, x\in E\}$ on a discrete state space $E$
and the capacity with respect to its Green's function $g(x, y)$.

It is known that for any finite set $A \subseteq E$,
there is a positive function $b(\cdot)$ supported by $A$ such that
\begin{equation}\label{Eq:Hittcap}
\P_x \big(X_t \in A\, \hbox{ for some }\, t \ge 0\big) =
\sum_{y\in A} g(x, y) b(y).
\end{equation}
This follows from Chung \cite{Chung73, Chung75}. For an explicit expression
of the function $b(\cdot)$, see also Syski \cite[p.435]{Syski92} or the
proof of Lemma~\ref{Lem:Cap} in the Appendix.

The natural capacity of $A$ with respect to $g$ is defined by
\begin{equation}\label{Eq:Hittcap2}
{\rm Cap}_g(A) = \sum_{y\in A} b(y).
\end{equation}
For any measure $\sigma$ on $A$, write $(g\sigma)(x)
= \sum_{y \in A} g(x, y) \sigma(y)$ for the potential due to the
charge~$\sigma$. Then we have
\begin{lemma}\label{Lem:Cap}
If the time homogeneous transient Markov chain $\{X_t,\, t \ge 0\}$
has a discrete state space $E$, is right continuous, and satisfies the
following conditions:
\begin{compactenum}[(i)]
  \item $p_t(x,y)\leq f(t)$ for all $x,y\in E$, where the function
  $f$ may depend on $(x,y)$ and is decreasing and integrable on $[0,\infty)$;
  \item for any $x\in E$, the rate $q_x$ of leaving $x$ is finite.
\end{compactenum}
Then
for any finite set $A \subseteq E$,
\begin{equation}\label{eq:capacity_max}
{\rm Cap}_g(A) = \max\Big\{ \sigma(A): \, \sigma \hbox{ is a
measure on }\, A \hbox{ such that }\, \max_{x \in A}
(g\sigma)(x) \le 1 \Big\}.
\end{equation}
In particular, \eqref{eq:capacity_max} holds for the VSRW.
\end{lemma}

It follows from (\ref{eq:capacity_max}) that the capacity of a
singleton $\{x\}$ is $g(x, x)^{-1}$.

As Lemma \ref{Lem:Cap} is (more or less) well known (but we didn't
succeed in finding a version similar to what is stated above that
fits our needs), we shall only briefly sketch its proof in the Appendix.

\subsection{Discrete fractal dimensions}
We recall briefly the definitions and basic properties of
fractal dimensions for subsets of $\Z^d$ from Barlow and Taylor \cite{BT92}.

For $x \in \Z^d$ and $n \ge 1$, define cubes
\begin{equation}\label{Eq:Cubes}
\begin{split}
C(x, n)&= \{y \in \Z^d: x_i \le y_i < x_i + n\}, \q\mbox{and}\\
V(x, n)&= \big\{y \in \Z^d: x_i - \frac 1 2 n \le y_i < x_i + \frac 1 2 n\big\}.
\end{split}
\end{equation}
Clearly $C(x, 1) = V(x, 1) = \{x\}$ and $\#\big(C(x, n)\big) =
\#\big(V(x, n)\big) = n^d$. Here and in the sequel, $\#(A)$ denotes the
cardinality of $A$.

Denote by ${\EuScript C}, {\EuScript C}_d$ and ${\EuScript C}_s$
the classes of cubes, dyadic cubes and semi-dyadic cubes in $\Z^d$. Namely,
\begin{equation}\label{Eq:Cubes2}
\begin{split}
{\EuScript C}&=  \big\{C(x, n): x \in \Z^d, n \ge 1\big\},\\
{\EuScript C}_d&=  \big\{C(x, 2^n): x \in 2^n \Z^d, n \ge 1\big\}, \q\mbox{and}\\
{\EuScript C}_s&=  \big\{V(x, 2^n): x \in 2^{n-1} \Z^d, n \ge 1\big\}.
\end{split}
\end{equation}

The side of $A \subseteq \Z^d$, denoted by $s(A)$, is defined by
\[
s(A) = \inf\big\{r > 0:\, A \subseteq C(x, r) \hbox{ for some } \,
x \in \Z^d\big\}.
\]
Let ${\EuScript C}_d^k$ and ${\EuScript C}_s^k$ denote the classes of dyadic
and semi-dyadic cubes of side $2^k$. Note that each $x \in \Z^d$ belongs
to a unique cube in ${\EuScript C}_d^k$, which is denoted by $Q_k(x)$.
Each $x \in \Z^d$ belongs to
$2^d$ cubes in ${\EuScript C}_s^k$ and we write $\widetilde V(x, 2^k)$ for the
semi-dyadic cube $V\in {\EuScript C}_s^k$ with center closest to $x$.

Let $V_n = V(0, 2^n)$ for all $n \ge 0$, $S_1 = V_1$ and $S_n = V_n \backslash V_{n-1}$
for $n \ge 2$. Thus $\{S_n, n \ge 2\}$ is a sequence of disjoint cubical shells
centered on the point $(-\frac 1 2, \ldots, - \frac 1 2)$.

Let $\cal H$ be the collection of functions $h: \R_+ \to \R_+$ such that $h$ is continuous,
monotone increasing, $h(0)=0$, and satisfies $h(2r) \le c_h\, h(r)$ for all $r \in [0, 1/2]$,
where $c_h$ is a constant. Functions in $\cal H$ are called measure functions.

For any $A \subseteq \Z^d$ and $h\in {\cal H}$, set
\begin{equation}\label{Def:Hmeasure0}
\nu_h(A, S_n) = \min\bigg\{\sum_{i=1}^k h\Big(\frac{s(B_i)} {2^n}\Big):
B_i \in {\EuScript C},\, A \cap S_n \subset \bigcup_{i=1}^k B_i\bigg\}.
\end{equation}
The discrete Hausdorff measure of $A$ with respect to the measure function $h$
is defined by
\begin{equation}\label{Def:Hmeasure}
m_h (A) = \sum_{n=1}^\infty \nu_h(A, S_n).
\end{equation}

If $h(r) = r^\alpha$ ($\alpha > 0$), we write $\nu_\alpha$ and $m_\alpha$
for $\nu_h$ and $m_h$, respectively.
The discrete Hausdorff dimension of $A$ is defined by
\begin{equation}\label{Def:dim}
\dim A = \inf\big\{\alpha> 0: m_\alpha (A) < \infty\big\}.
\end{equation}

It is often more convenient to replace ${\EuScript C}$ in (\ref{Def:Hmeasure0}) by
the smaller class ${\EuScript C}_d$, the corresponding values to (\ref{Def:Hmeasure0})
and (\ref{Def:Hmeasure}) will be written as $\widetilde{\nu}_h (A, S_n)$
and $\widetilde{m}_h(A)$, respectively. Barlow and Taylor \cite[p.128]{BT92}
proved that $\nu_h(A, S_n)
\le \widetilde{\nu}_h (A, S_n) \le 2^d \nu_h(A, S_n)$. Hence ${m}_h(A)$ and
$\widetilde{m}_h(A)$ are comparable, and replacing $m_\alpha $ in \eqref{Def:dim}
by $\widetilde{m}_\alpha$ defines the same $\dim A$.

Discrete packing measure and packing dimension of $A$ are defined in a dual way.
For any $h \in {\cal H}$ and $\ep>0$, define
\begin{equation}\label{Def:pmeasure0}
\tau_h(A, S_n, \ep) = \max\bigg\{\sum_{i=1}^k h\Big(\frac{r_i} {2^n}\Big):
x_i \in A \cap S_n, V(x_i, r_i) \hbox{ disjoint,}\ 1 \le r_i \le 2^{(1-\ep)n} \bigg\}
\end{equation}
and
\begin{equation}\label{Def:pmeasure}
p_h(A, \ep) = \sum_{n=1}^\infty \tau_h(A, S_n, \ep).
\end{equation}
A set $A \subseteq \Z^d$ is said to be \emph{$h$-packing finite} if $p_h(A, \ep)< \infty$
for all $\ep \in (0, 1)$. Again, if $h(r) = r^\alpha$, we write $\tau_h$ and $p_h$ as
$\tau_\alpha$ and $p_\alpha$.

The discrete packing dimension
of $A$ is defined by
\begin{equation}\label{Def:dimp}
\dimp A = \inf\big\{\alpha> 0: A \hbox{ is $r^\alpha$-packing finite} \big\}.
\end{equation}
One can use semi-dyadic cubes in (\ref{Def:pmeasure0}) and define
\begin{equation}\label{Def:pmeasures}
\widetilde{\tau}_h(A, S_n, \ep) = \max\bigg\{\sum_{i=1}^k h\Big(\frac{2^{k_i}} {2^n}\Big):
x_i \in A \cap S_n, \widetilde{V}(x_i, 2^{k_i}) \hbox{ disjoint,}\ 2^{k_i}
\le 2^{(1-\ep)n} \bigg\}
\end{equation}
and the corresponding $\widetilde{p}_h(A, \ep)$. \cite[p.130]{BT92}
proved that there exists a constant $c>1$ (depending on $h$) such that
$$
c^{-1}\tau_h(A, S_n, \ep) \le \widetilde{\tau}_h(A, S_n, \ep)
\le c\, \tau_h(A, S_n, \ep)
$$
for all $A, n $ and $\ep \in (0, 1)$. Thus, $A$ is $h$-packing finite if and only if
$\widetilde{p}_h(A, \ep)< \infty$ for all $\ep \in (0, 1)$.

From (\ref{Def:dim}) and (\ref{Def:dimp}) it is clear that $\dim A$ and $\dimp A$ do not
depend on the part of~$A$ which lies inside any ball of finite radius. They are determined
by the geometric structure of~$A$ at infinity.
Similarly to the ordinary Hausdorff and packing dimensions in $\R^d$, the discrete Hausdorff and
packing dimensions on $\Z^d$ satisfy the following relationship: For all $A \subseteq  \Z^d$,
\begin{equation}\label{dim-dimp}
0 \le \dim A \le \dimp A \le d,
\end{equation}
and the inequalities may be strict. See Barlow and Taylor
\cite[pp.130, 132 and 136]{BT92}.

It is often not difficult to find optimal covering or packing for $A \cap S_n$,
which leads to good upper bound for $\nu_h(A\cap S_n)$ and lower bound
for $\tau_h(A, S_n,\ep)$. However, a direct approach for obtaining the
lower bound for $\nu_h(A\cap S_n)$ (or upper bound for $\tau_h(A, S_n,\ep)$)
is usually tricky. The following lemmas are useful. The first is an analogue
of the density lemma and is a consequence of Theorem 4.1 of Barlow and Taylor \cite{BT92}. The
second is an analogue of Frostman's lemma and follows from Theorem 4.2 of Barlow and Taylor
\cite{BT92}.
\begin{lemma}\label{Lem:density}
Let $h \in {\cal H}$ and $\mu$ be a measure on $A \subseteq S_n$. If
\begin{equation*}\label{density}
\mu(A \cap V(x, 2^k)) \le a_1 \, h\big(2^{k-n}\big)\quad \ \hbox{ for all } x \in \Z^d,
\ 0 \le k \le n.
\end{equation*}
Then $ \nu_h (A,\, S_n) \ge 2^{-d} a_1^{-1}\mu(A)$.
\end{lemma}

\begin{lemma}\label{Lem:Frostman}
Let $h \in {\cal H}$ and $A \subseteq S_n$. Then there is a measure $\mu$
on $A$ that satisfies
\begin{equation*}\label{Eq:Frostman0}
\mu(A) \ge \nu_{h} (A, S_n)\ \ \ \ \hbox{and }\ \ \
\mu(V(x, 2^k)) \le 2^d\, h\Big(\frac {2^k} {2^n}\Big)\
\hbox{ for all } 0 \le k \le n,\, x \in A.
\end{equation*}
\end{lemma}

\subsection{A SLLN for dependent events}

The increments of VSRW are not independent. For this reason, we here establish
a SLLN for dependent events which will be used in our proof of Theorem \ref{Th:main}
below.

\begin{lemma}\label{Lem:SLLN}
Suppose that $\{A_i\}, \{B_i\}$ are two sequences of events adapted to a (common)
filtration $\{\sF_i\}$ and are such that for some positive
constants $p, a$ and $\delta$
\begin{equation}\label{Eq:SLLN1}
   \P(A_{i+1} | \sF_{i})\geq p \ \ \mbox{ on event } B_i, \mbox{ and }\
  \P(B_i^c)\leq ae^{-\delta i}\  \mbox{ for all }\, i.
\end{equation}
Write $X_i= \I_{A_i}$, and $S_n = \sum_{i=1}^n X_i$. Then there exists $\vep>0$ such that
\[
  \liminf_{n \to \infty} \frac{S_n}{n} \geq  \vep \quad \mbox{ almost surely.}
\]
\end{lemma}
\begin{proof}\,
We first estimate the moment generating function of $S_n$. For any $t>0$,
\begin{equation}\label{eq:mgf_Sn_ind}
\aligned
\E(e^{-t S_n})&=\E \big( \E(e^{-t X_n} | \sF_{n-1}) \cdot e^{-t S_{n-1}}\big)\\
  &\leq
  \E\big(\E(e^{-t X_n} | \sF_{n-1})\cdot\I_{B_{n-1}} \cdot e^{-t S_{n-1}}\big)
  +\P(B_{n-1}^c).
\endaligned
\end{equation}
By using the first inequality in (\ref{Eq:SLLN1}) and the elementary inequality $1- x \le e^{-x}$
($x \ge 0$), we derive that
\begin{equation}\label{eq:mgf_Sn_bd}
\E\big(e^{-t X_n} | \sF_{n-1}\big)\leq q(t):= e^{-p(1-e^{-t})}<1 \q \mbox{ on event } B_{n-1}.
\end{equation}
Now choose  $k>0$ large enough and  $b>0$ small enough such that
\[
   b\leq \delta,\q ke^{-b}\geq 1, \q\mbox{and}\q q(t) e^b + \frac{ae^{\delta}}{k}\leq 1.
\]
We go on to show that
\begin{equation}\label{eq:mgf_Sn}
\E(e^{-t S_n}) \leq k e^{-b n},\q\mbox{for all } n.
\end{equation}
In fact, by the choices of $k$ and $b$, \eqref{eq:mgf_Sn} holds automatically for $n=1$.
Now suppose that it holds for $n-1$, then by \eqref{eq:mgf_Sn_ind} and \eqref{eq:mgf_Sn_bd}
and using induction one gets that
\[
\E(e^{-t S_n})\leq q(t)   \E(e^{-t S_{n-1}}) + \P(B_{n-1}^c)
  \leq q(t) k e^{-b (n-1)} + a e^{\delta}\cdot e^{-\delta n}.
\]
The last term is bounded by $k e^{-b n}$, by the choices of $k$ and $b$.

Once we have the bound \eqref{eq:mgf_Sn} for $\E(e^{-t S_n})$, the conclusion
then follows easily by using the Chebyshev's inequality and the Borel-Cantelli lemma.
\end{proof}

\section{Proofs of Theorems \ref{Th:main} and \ref{Th:main2}}

As in Barlow and Taylor \cite{BT92}, the proof of Theorem \ref{Th:main} is divided
into proving the upper bound $\dimp {\mathrm R} \le 2$\, $\P_0^\o$-a.s.
and the lower bound $\dim {\mathrm R} \ge 2$\, $\P_0^\o$-a.s., separately.
The upper bound is proved by using a first moment argument and the lower
bound is proved by using a ``mass distribution'' method. However, since there
are significant differences between VSRW and the strictly stable random walks in
Barlow and Taylor \cite{BT92}, some preparations are needed.

In the following we establish quenched results on hitting probability, sojourn time,
maximal inequality, and a zero-one law
for VSRW. These results may also be useful for studying other
properties of VSRW.

\subsection{Hitting probability estimates}

We start with the following lemma. Its proof is a slight modification
of that of Proposition 2.1 in Xiao \cite{X98} which is an extension of Theorem 1
in Khoshnevisan \cite{Kh97} for L\'evy processes.

\begin{lemma}\label{Lem:Hit0}
Let $\{X_t,\, t \ge 0, \P_x, x\in \Z^d\}$ be a time homogeneous (continuous time)
Markov chain. 
Then for any $x, y \in \Z^d$, $b > a \ge 0$ and $r > 0$,
\begin{equation} \label{Eq:Hit0}
\begin{split}
\frac 1 2 \frac{\int_a^b \P_x \big( X_t \in C(y, r)\big)\, dt }
{\sup_{z \in C(y,r)} \int_0^b \P_z \big( X_t \in C(y, r)\big)\, dt}
& \le \P_x\Big( X_t \in C(y, r)\ \hbox{ for some } \, a \le t \le b\Big)\\
&\qquad \le \frac{\int_a^{2b-a} \P_x \big( X_t \in C(y, r)\big)\, dt}
{\inf_{z \in C(y,r)} \int_0^{b-a} \P_z\big( X_t \in C(y, r)\big)\, dt}.
\end{split}
\end{equation}
\end{lemma}

Observe that, if $\int_0^\infty \P_x \big(X_s \in C(y, r)\big)ds < \infty$,
then we can take $a=0$ and $b=\infty$ in Lemma \ref{Lem:Hit0}.

Now we apply Lemma \ref{Lem:Hit0} to derive
the following hitting probability estimates for the VSRW~$X$.

\begin{lemma}\label{Lem:Hit}
Assume that $d \ge 3$ and $\PP\big(\mu_e \ge 1\big)= 1$. Then $\PP$-a.s.
for all $n$ large enough,
$r \in [n^{3d/2},\ 2^{n-1}]$, $y \in V(0, 2^n)$, and all $x\in \Z^d$
such that $\d{x-y}\ge 2 r$
we have
\begin{equation} \label{Eq:Hit}
\P^\o_x\Big( X_t \in C(y, r)\ \hbox{ for some } \, t > 0\Big)
\asymp \bigg(\frac r {|x-y|}\bigg)^{d-2}.
\end{equation}
\end{lemma}
Here and in the sequel, $f \asymp g$ means that the ratio $f/g$
is bounded from below and above by positive and finite constants
which are independent of the variables involved ($x, y$ and $r$ in this case).

\begin{proof}\,
We will apply Lemma \ref{Lem:Hit0} with $a=0$ and $b=\infty$.
We first consider the denominators in (\ref{Eq:Hit0}) and show that
there exists a constant  $c>1$ such that
\begin{equation}\label{Eq:Hit3}
c^{-1} r^2 \le \int_0^\infty \P_z^\o \big(X_s \in  C(y, r)\big)\,ds \le c \,r^2
\end{equation}
for all $z \in C(y, r)$
and for all $y$ and $r$ that we consider.

By Fubini's theorem, we write the integral in \eqref{Eq:Hit3} in terms of the
Green's function of $X$:
\begin{equation}\label{Eq:Hit4}
\begin{split}
\int_0^\infty \P_z^\o \big(X_s \in  C(y, r)\big)\,ds &= \sum_{w \in C(y, r)} g^\o(z, w). \\
\end{split}
\end{equation}

By \eqref{Eq:n0}, for all
$n$ sufficiently large and for all $r\leq 2^{n-1}$,
\begin{equation}\label{eq:sep}
\max_{y \in V(0, 2^{n}),\ z \in C(y, r)}U_z \le c_{11}\, n^3 \leq n^{3d/2}/4.
\end{equation}
Hence we can apply (d)
and (e)
of Lemma \ref{Lem:BD} to (\ref{Eq:Hit4}) and obtain that for
every $z \in C(y, r)$ with $r \ge n^{3d/2}$,
\begin{equation}\label{Eq:Hit5}
\begin{split}
\int_0^\infty \P_z^\o \big(X_s \in  C(y, r)\big)\,ds
&\le \sum_{\{|w-z|\le c_{11} n^3\}} c_{10}  + \sum_{\{c_{11} n^3
\le |w-z|\le \sqrt{d}\, r\}} \frac{c_9} {|z- w|^{d-2}} \\
&\le  c\, \big( n^{3d} +  r^2\big) \le c\, r^2.
\end{split}
\end{equation}
This proves the upper bound in \eqref{Eq:Hit3}. On the other hand,
since we only consider large $r$'s,
for any $z\in C(y,r)$,
$$
\#\Big\{w \in C(y, r): \frac r 4 \le |w-z|\le \frac{3r} 4 \Big\} \asymp r^d.
$$
Denote the above set by $\Gamma$. Then by using \eqref{Eq:Hit4},
\eqref{eq:sep} and (d) of Lemma \ref{Lem:BD} again, we have
\begin{equation}\label{Eq:Hit5b}
\begin{split}
\int_0^\infty \P_z^\o \big(X_s \in  C(y, r)\big)\,ds \ge
\sum_{w \in \Gamma} \frac{c_8} {|z- w|^{d-2}} \ge  c^{-1}\, r^2,
\end{split}
\end{equation}
which proves the lower bound in \eqref{Eq:Hit3}.

To estimate the numerators in (\ref{Eq:Hit0}),
noting that for all $w\in C(y,r)$, since
$\d{x-y}\geq 2r$ and hence
$|x-w|\geq \d{x-w}\geq r\geq \max_{w\in C(y, r)} U_w$,
we can use again (d) of Lemma
\ref{Lem:BD} to get
\begin{equation}\label{Eq:Hit6}
\begin{split}
\int_0^{\infty} \P_x^\o \big(X_t \in C(y, r)\big)\, dt
&= \sum_{w \in C(y, r)} g^\o(x, w)  \\
&\asymp \sum_{w \in C(y, r)} \frac{1} {|x- w|^{d-2}} \asymp \frac{r^{d}}{|x-y|^{d-2}},
\end{split}
\end{equation}
where in the last step we used again that $|x-y|\geq \d{x-y} \geq 2r.$
Hence (\ref{Eq:Hit}) follows from Lemma \ref{Lem:Hit0}, \eqref{Eq:Hit3} and
\eqref{Eq:Hit6}.
\end{proof}

Similarly, for the discrete time VSRW $\{\wh X_n,\,n\geq 0\}$, we have the
following estimate regarding the hitting probabilities.

\begin{lemma}\label{Lem:hit_pt}
Assume that $d \ge 3$ and $\PP(\mu_e \ge 1)= 1$. Then $\PP$-a.s. for all
$n$ large enough, for all $x\in V(0,2^{n-2})$, and for all $y\in S_n
(=V(0,2^{n})\backslash V(0,2^{n-1}))$,
\begin{equation}
\P^\o_x\Big( \wh X_i = y\ \hbox{ for some } \, i \geq 0\Big)
\geq \frac{c_8}{c_{10}|x-y|^{d-2}}.
\end{equation}
\end{lemma}
\begin{proof}\
By the strong Markov property,
\[
   \wh{g}^\o(x,y)= \P^\o_x\Big( \wh X_i = y\ \hbox{ for some } \, i
   \geq 0\Big)\cdot \wh{g}^\o(y,y).
\]
The conclusion then follows from Lemma \ref{Lem:green_discrete} and \eqref{Eq:n0}.
\end{proof}

Our next lemma, which is similar to Proposition 8.1 in Barlow and Taylor \cite{BT92},
establishes a connection between the capacity Cap$_{g^\o}$ associated
with VSRW $X$ and Hausdorff measures.
\begin{lemma}\label{CapHau}
Assume that $d \ge 3$ and $\PP(\mu_e \ge 1)= 1$. Then there exists a constant
$c_{13} \ge 1$ such that for $\PP$-almost every
$\o \in \Omega$, all $n$ large enough and sets $A \subseteq S_n$,
\begin{equation}\label{Eq:CapHau}
c_{13}^{-1}\, 2^{n(d-2)}\, \nu_{h_2} (A, S_n) \le {\rm Cap}_{g^\o}(A) \le
c_{13}\, 2^{n(d-2)}\, \nu_{h_1} (A, S_n).
\end{equation}
In the above,
\[
h_1(r) =
\left\{
 \aligned
 &r^{d-2}\bigg(\log \Big(\frac 1 r\Big)\bigg)^{3d(d-2)/2}, \q & \mbox{if } r \leq r_0;\\
 & r^{d-2}\bigg(\log \Big(\frac 1 r_0\Big)\bigg)^{3d(d-2)/2}, \q & \mbox{if } r \geq r_0,
 \endaligned
\right.
\]
where $r_0=\exp(-3d/2)$ is such that $h_1(\cdot)$ is monotone increasing;
and
\[
h_2(r) = r^{d-2}\bigg(\log \Big(\frac 1 r\Big)\bigg)^{- c_{14}},
\]
where
$c_{14} > (3/\log 2 + 1)(d-2)$
is a constant.
\end{lemma}
\begin{proof}\,
Let $\{B_i, 1 \le i \le m\}$ be an optimal cover for $A$, in the sense that
\[
  \nu_{h_1} (A, S_n) =  \sum_{i=1}^m h_1\bigg(\frac{s(B_i)} {2^n}\bigg).
\]
Write $r_i=s(B_i)$. If $r_i \ge  n^{3d/2}$, then by Lemma \ref{Lem:Hit}, (d) of
Lemma \ref{Lem:BD} and \eqref{Eq:n0}, and  the definitions
\eqref{Eq:Hittcap} and \eqref{Eq:Hittcap2} of capacity, one can get that
${\rm Cap}_{g^\o}(B_i\cap S_n) \le c r_i^{d-2}$. On the other hand, if
$r_i <  n^{3d/2}$, then we enlarge the cube so that its side $r_i' =
[n^{3d/2}] + 1$, which has capacity bounded by
$ c (r_i')^{d-2}\leq c n^{3d(d-2)/2}\leq c (\log (2^n/r_i))^{3d(d-2)/2}.$
It follows that for some constant $c_{13}>0$, for all $n$ sufficiently large,
\[
\begin{split}
{\rm Cap}_{g^\o}(A) \le \sum_{i=1}^m {\rm Cap}_{g^\o}(B_i\cap S_n)
\le c_{13}\, 2^{n(d-2)}  \sum_{i=1}^m h_1\bigg(\frac{r_i} {2^n}\bigg)
= c_{13}\, 2^{n(d-2)}\nu_{h_1} (A, S_n).
\end{split}
\]

Next we prove the lower bound in (\ref{Eq:CapHau}). By Lemma \ref{Lem:Frostman} there
is a measure $\mu$ on $A$ such that
\begin{equation}\label{Eq:Frostman2}
\mu(A) \ge \nu_{h_2} (A, S_n) \ \ \hbox{ and }\ \
\mu(V(x, 2^k)) \le 2^d\, h_2\Big(\frac {2^k} {2^n}\Big)\ \
\hbox{ for all } 0 \le k \le n,\ x \in A.
\end{equation}
For any $x \in A$, let $S_k(x) = V(x, 2^k)\backslash V(x, 2^{k-1})$. Let
$c_0=3/\log 2 + 1$ so that for all $n$ large enough,
$2^{c_0\log n}\geq c_{12} n^3$. Then by (d) and (e) of Lemma \ref{Lem:BD} and \eqref{Eq:n0},
\begin{equation}\label{Eq:CapHau2}
\begin{split}
(g^\o \mu)(x)
& \leq \sum_{k=0}^{2n} \sum_{y \in A\cap S_k(x)} g^\o(x, y) \mu(y)\\
&\le \sum_{k=0}^{c_0\log n} c_{10}\mu\big(S_k(x)\big) +
\sum_{k= 1+c_0\log n}^{2n} \frac{c_{9}}{2^{k(d-2)}} \mu\big(S_k(x)\big).
\end{split}
\end{equation}
By using the second inequality in \eqref{Eq:Frostman2},
and noting that $c_{14} > c_0(d-2)$
one can verify that
\[
\sum_{k=0}^{c_0\log n} c_{10}\mu\big(S_k(x)\big) \le c \, 2^{-n(d-2)} \ \hbox{ and }\
\sum_{k= 1+c_{0}\log n}^{2n} \frac{c_{9}}{2^{k(d-2)}} \mu\big(S_k(x)\big) \le c \, 2^{-n(d-2)}.
\]
This and (\ref{Eq:CapHau2}) imply $(g^\o \mu)(x) \le c_{15} 2^{-n(d-2)}$ for all $x \in A$.
Now we take the measure $\mu' = c_{15}^{-1} 2^{n(d-2)}\mu$. Then $(g^\o \mu')(x) \le 1$
for all $x \in A$. Therefore, by (\ref{eq:capacity_max}) and the first inequality
in \eqref{Eq:Frostman2},
\[
{\rm Cap}_{g^\o}(A) \ge \mu' (A) = c_{15}^{-1} 2^{n(d-2)}\mu(A) \ge c_{15}^{-1}
2^{n(d-2)}\,\nu_{h_2} (A, S_n).
\]
This proves the lower bound in (\ref{Eq:CapHau}).
\end{proof}

\subsection{Tail probability of the sojourn measure for the discrete time VSRW}

In this subsection we focus on the discrete time VSRW $\{X_n,\, n=0,1,\ldots,\}$.
For this process, for any $F\subseteq \Z^d$, the sojourn time of $F$
is defined by
\begin{equation}\label{Def:Sojourn}
T(F) = \#\{n\geq 0: X_n\in F\}.
\end{equation}

The following lemma is an analogue of Lemma 7.6 in Barlow and Taylor \cite{BT92}
for random walks, which holds for any time-homogeneous Markov chains and can
be proved similarly as in Lemma 3.1 in Pruitt and Taylor \cite{PT69}.
\begin{lemma}\label{Lem:Sojourn}
If $F\subseteq \Z^d$ satisfies
\begin{equation}\label{dfn:max_sojourn}
M(F):=\sup_{y\in F} \E_y (T(F)) \in (0,\infty).
\end{equation}
Then for every $\delta \in (0, 1)$
and any $\lambda\geq 0$, and for all $x \in \Z^d$,
\begin{equation*}\label{Eq:Sojourn}
\P_x\Big( T\big(F\big) \ge \lambda  M(F) \Big) \le e^{-\de \lambda}.
\end{equation*}
\end{lemma}
\begin{proof}
\ It suffices to prove that for all $x \in \Z^d$ and all integers
$m \ge 1$,
\begin{equation}\label{Eq:Smom}
\E_x \big( T(F)^m\big) \le m! (M(F))^m.
\end{equation}
This can be verified by using induction and the Markov property in a standard
way (see, e.g., \cite{PT69}). We omit the details.
\end{proof}

Next we estimate $M(F)$. Denote by $V_k(y)=V(y, 2^k)$  the cube in $\Z^d$ centered
at $y$ with side $2^k$. Let $c_{17}$ be a large constant so that
\begin{equation}\label{Eq:c17}
2^k\geq 2c_{11}n^3\q\mbox{and}\q 2^{2k}\geq n^{3d} \q
\mbox{for all }k\geq c_{17}\log n
>0.
\end{equation}
\begin{lemma}\label{Lem:S2}
Assume that $d \ge 3$ and $\PP(\mu_e \ge 1)= 1$.
Then there exists a constant $c_{16}$  such that
$\PP$-a.s. for all $n$ large enough,
and $ c_{17} \log n\le k\le n $ the inequality
\begin{equation} \label{Eq:S2}
\E^\o_x\big( T(V_k(y)) \big) \le c_{16}\,  2^{2k}
\end{equation}
holds uniformly for all $x, y \in V(0,2^{n+1})$.
\end{lemma}
\begin{proof}\,
As in (\ref{Eq:Hit4}), we have that for any $F \subseteq \Z^d$ and $x \in \Z^d$,
\begin{equation*} 
\E^\o_x\big( T(F) \big) = \sum_{z \in F}\wh{g}^\o(x, z).
\end{equation*}
Moreover, by \eqref{Eq:n0}
$\max_{x\in V_{n+1}} U_x\leq c_{11}\, n^3.$
It follows from Lemma \ref{Lem:green_discrete} that
\begin{equation*}
\begin{split}
\E^\o_x\big( T(V_k(y)) \big)
&\le \sum_{z \in  V(x, c_{11} n^3)} \wh{g}^\o(x, z) +
\sum_{z \in V(y, 2^k)\backslash V(x, c_{11} n^3)} \frac {c_9} {|x-z|^{d-2}}  \\
&\le C\,\big( n^{3d} +  2^{2k}\big)
\le c_{16}\,  2^{2k}.
\end{split}
\end{equation*}
This proves \eqref{Eq:S2}.
\end{proof}

It follows from Lemma \ref{Lem:S2} that $\PP$-a.s. for all $n$ large enough,
$c_{17} \log n\le k \le n  $ and all $x\in V_{n}$,
\begin{equation} \label{Eq:M_F}
M(Q_k(x)) \le c_{16}\,  2^{2k},
\end{equation}
where, recall that, for any $x\in\Z^d$ and any $0\leq k\in \Z$, $Q_k(x)$
is the unique cube in $\ESC_d^k$ that contains $x$.

\subsection{A maximal inequality}

The following lemma estimates the tail probability of the maximal
displacement of VSRW $X$.

\begin{lemma}\label{Lem:Maximal}
Assume that $d \ge 3$ and $\PP(\mu_e \ge 1)= 1$. Then
there exist constants $c_{18}$, $c_{19}$ and $c_{20}$ such that $\PP$-a.s.
for all $n$ large enough ($n\ge n_0$)
the inequality
\begin{equation}\label{Eq:M}
\P_x^\o\bigg(\sup_{0 \le t \le T}|X_t- x| > \lambda \sqrt T\bigg)
\le c_{18}\,\exp\big(-c_{19}\lambda^2\big)
\end{equation}
holds for all $x \in V(0, 2^{n-1})$, $(c_{11} n^3)^2\le T \le 2^{n}$
and $c_{20} \le \lambda < \sqrt{T}/2$.
\end{lemma}

\begin{proof}\, For any $n$ and  $T, a> 0$, let
\begin{equation}\label{Def:alpha}
\alpha (T, a) = \sup_{\d{x}\leq 2^n,\, 0\le t \le T} \P_x^\o \big(|X_t-X_0|> a\big).
\end{equation}
For any $x \in V(0, 2^{n-1}) $ and $ 2\leq M < 2^{n-1}$, we consider the
stopping time
\[
\tau = \inf\{t>0: \d{X_t-  x} > M \}.
\]
For VSRW $X$ started at $x$, we have
$\d{X_{\tau}- x} \le M +1$
and hence $\d{X_\tau}\leq 2^{n-1} + M + 1\leq 2^n. $
The triangle inequality and the strong Markov property imply that
\begin{equation}\label{Eq:Ott0}
\P_x^\o\bigg( |X_T - x| > \frac M 2 \bigg) \ge
\E_x^\o\bigg[ \P_{X_\tau}^\o\bigg(|X_{T-\tau} - X_0| \leq \frac M 2 \bigg)\,
\I_{\{\tau \le T\}}\bigg].
\end{equation}
By the definition \eqref{Def:alpha} and that  $\d{X_\tau}\leq 2^{n}$,
for any pair $(\tau, X_\tau)$,
\begin{equation}\label{Eq:Ott1}
\begin{split}
\P_{X_\tau}^\o\bigg(|X_{T-\tau} - X_0| \leq \frac M 2 \bigg)
&= 1 - \P_{X_\tau}^\o\bigg(|X_{T-\tau} - X_0| > \frac M 2 \bigg)\\
&\ge 1- \alpha (T, M/2).
\end{split}
\end{equation}
Hence we derive the following quenched Ottaviani-type inequality:
For all $x \in V(0, 2^{n-1}) $, $ 2\leq M < 2^{n-1}$
and $T >0$  such that $\a(T, M/2) <1$,
\begin{equation}\label{Eq:Ott}
\P_x^\o\bigg( \sup_{0 \le t\le T} |X_t - x| > M \bigg)
\le
\frac{\P_x^\o\big(|X_T - x| > M/2 \big)} {1 - \a(T, M/2)}.
\end{equation}
This is reminiscent to Lemma 2 in Gikhman and Skorohod \cite[p.420]{GS74}.

Next, recall that $\PP$-\hbox{a.s.}  for all $n\ge n_0$, we have
$\max_{\d{x}\leq 2^n} U_x \le c_{11} n^3$,
see \eqref{Eq:n0}. If $a > c_{11} n^3$, then we can apply (b) of Lemma \ref{Lem:BD} to
deduce that for all $x \in V(0, 2^n)$ and all $0 \le t \le T$ with $t \le a$,
\begin{equation}\label{Eq:M2}
\begin{split}
\P_x^\o\big(|X_t- x|> a\big)&= \sum_{y \in \Z^d:|y-x|> a} p^\o_t(x, y)\\
&\le c_4\sum_{y \in \Z^d:|y-x|> a} \exp\Big(- c_5|y-x|\Big)\\
&\le c_{21}\,e^{-c_{22} a}.
\end{split}
\end{equation}
If $T > a \, (> c_{11} n^3)$ and $a <t \le T$, then it can be verified that
for all $x \in V(0, 2^{n-1})$,
\begin{equation}\label{Eq:M3}
\begin{split}
\P_x^\o\big(|X_t- x|> a\big)
&= \sum_{a <|y-x|\le t} p^\o_t(x, y) + \sum_{ |y-x|> t} p^\o_t(x, y)\\
&\le c_4\sum_{a < |y-x| \le t} t^{-d/2} \exp\Big(- c_5|y-x|^2/t\Big) + c_{21} e^{-c_{22} t}\\
&\le
c_{23} e^{-c_{24} (a/\sqrt{t})^2}
+ c_{21} e^{-c_{22} t}.
\end{split}
\end{equation}

Now we apply  \eqref{Eq:Ott} with $M = \lambda \sqrt T$, where $(c_{11} n^3)^2
\le T < 2^n $ and $c_{20} \le \lambda \le \sqrt T/2$. It follows from \eqref{Eq:M2}
and \eqref{Eq:M3} that we can choose $n$ and the constant $c_{20}$ large enough
such that
\begin{equation}\label{Eq:M4}
\alpha(T,\, M/2)\le \frac 1 2.
\end{equation}
By (\ref{Eq:M3}),
 we have
that for all $\lambda \le \sqrt{T}/2,$
\begin{equation}\label{Eq:M5}
\P_x^\o\Big(|X_T - x| >  \lambda \sqrt T\Big)
\le c'
e^{-c'' \lambda^2}.
\end{equation}
Plugging \eqref{Eq:M4} and \eqref{Eq:M5} into \eqref{Eq:Ott} yields (\ref{Eq:M}).
\end{proof}

It is known that VSRW spends a time of order $n^2$ in the cube $V(0,n)$
(Barlow and \u Cern\' y \cite[p.655]{BC11}). By applying Lemma \ref{Lem:Maximal}
and the Borel-Cantelli lemma we obtain
\begin{cor}\label{cor:LIL} For $\bP$-\hbox{a.e.} $\o$, $\P_0^\o\hbox{-a.s.}$,
\[
\limsup_{T \to \infty}\frac{\max_{0\le t \le T}|X_t|} {\sqrt{T \log \log T}}
\le \frac{1}{\sqrt{c_{19}}}.
\]
\end{cor}
Consequently, the time that VSRW $X$ spends in the cube $V(0, n)$ is at least
$c n^2/\sqrt{\log \log n}$.

\subsection{A zero-one law }
\begin{lemma}\label{01_law} For any (infinite)
set $A \subset \Z^d$,  for $\PP$-almost every $\o \in \Omega$,
\[
\P^\o_0\big(X_t \in A\ \hbox{ for  arbitrarily large } t >0\big) \in \{0,1\}.
\]
\end{lemma}
\begin{proof}
This is a consequence of an elliptic Harnack inequality that the VSRW $X$ satisfies.
More explicitly, define
\[
 h(x)=\P^\o_x\big(X_t \in A\ \hbox{ for  arbitrarily large } t >0\big).
\]
Then $h$ is a harmonic function (with respect to the generator ${\cal L}_V$ in
\eqref{Def:generator}). Let $a=\inf_x h(x)\ge 0$, and let $g= h-a$, so $g \ge 0$
with $\inf_x g(x) =0$.
If $g$ is not identically zero, then there exists $x_0$ such that $g(x_0)>0.$
Now for any $R\geq U_{x_0}$, by Corollary 4.8 in Barlow and Deuschel \cite{BD10},
\[
 g(x_0) \le  \sup_{x\in B(x_0,R/2)} g(x) \le C \inf_{x\in B(x_0,R/2)} g(x).
\]
This holds for all $R\geq U_{x_0}$, so one gets
\[
 g(y) \ge  g(x_0)/C \mbox{ for all  } y,
 \]
a contradiction to that $\inf_x g(x) =0$. So $h$ must be a constant function.
On the other hand,
the martingale convergence theorem tells us that  $\P^\o_0$ almost surely,
\[
h(X_t) \to \I_{ \{X_t \in A\ \hbox{ for  arbitrarily large } t >0\}} \quad \mbox{as  } t\to \infty.
\]
So $h$ is either constantly 0 or constantly 1.
\end{proof}

\subsection{Proofs of Theorems \ref{Th:main} and \ref{Th:main2}}

With the results established above, we are ready to prove Theorems \ref{Th:main}
and \ref{Th:main2}. Even though the arguments are similar to the proofs of
Theorem 7.8 and Theorem 8.3 of Barlow and Taylor \cite{BT92}, several modifications
are needed.

\begin{proof}{\it of Theorem \ref{Th:main}}\,
Firstly we prove that for $\PP$-\hbox{a.e.} $\o$,
the packing dimension of the range
$\dimp {\rm R}\le 2$\, $\P_0^\o$-a.s.
This is done by using a first moment argument.

Let $M_k$ be the total number of semi-dyadic cubes in ${\EuScript C}^k_s$ of order $k$
which are contained in $S_n$ and are hit by $\{X_t, t \ge 0\}$. Since there are at
most $c 2^{(n-k)d}$ semi-dyadic cubes of order $k$ contained in $S_n$ and, by
Lemma \ref{Lem:Hit},
for all $k$ such that $n^{3d/2} \leq 2^k\leq 2^{n-1}$, or equivalently,
$c\log n\leq k\leq n-1$ for some $c>0,$
each of them can be hit by $X$ with probability at most $c 2^{(d-2)(k-n)}$.
Hence
\begin{equation}\label{Eq:prep1}
\E_0^\o \big(M_k\big)\le c\, 2^{2(n-k)},\q\mbox{for all } c\log n\leq k\leq n-1.
\end{equation}
Now we take $\delta >0$ small, and
$\beta = 2 +  \delta$.
It follows from
(\ref{Def:pmeasures}) and \eqref{Eq:prep1} that for $\ep \in (0, 1)$,
\begin{equation}\label{Eq:prep2}
\begin{split}
\E_0^\o\Big(\widetilde{\tau}_\beta\big({\rm R}, S_n, \ep \big)\Big)
&\le  \sum_{k=1}^{c\log n} c2^{(n-k)d}\cdot 2^{(d-2)(c\log n-n)} \cdot
\Big(\frac{2^k}{2^n}\Big)^{\beta}
    + c\sum_{k=c\log n}^{n (1-\ep)} 2^{2(n-k)}\Big(\frac{2^k}{2^n}\Big)^{\beta}\\
&\le c\, 2^{(d-2)c\log n}\cdot 2^{-n \ep \delta}.
\end{split}
\end{equation}
Hence $\widetilde{p}_{\beta}\big({\rm R}, \ep)< \infty$ for all $\ep \in (0, 1)$
\, $\P_0^\o$-a.s.
This and the arbitrariness of $\delta>0$ imply $\dimp {\rm R}\le 2$\, $\P_0^\o$-a.s.

Secondly we prove $\dim {\rm R}\ge 2$ $\P_0^\o$-a.s.
Let $\wh{{\rm R}}$ be the range of the discrete time VSRW $\wh{X}$:
\[
\begin{split}
  \wh{\rm R}&:= \big\{x\in\Z^d: \wh{X}_n = x\quad\mbox{for some }   n\geq 0\big\}\\
  &\ = \{x\in\Z^d: {X}_n = x\quad\mbox{for some }
  n\geq 0\}.
  \end{split}
\]
Since $\wh{\rm R}$ is a subset of R, it suffices to show that $\dim {\wh{\rm R}}\ge 2$
$\P_0^\o$-a.s.
Let $\mu$ be the measure on~$\wh{\rm R}$ which assigns mass 1 to each point of $\wh{\rm R}$.
We claim that there is a constant $c_{27}$
such that $\P_0^\o$-a.s. for all $n$ large enough
\begin{equation}\label{Eq:claim}
\mu\big(Q_k(x)\big) \le c_{27}\, n\, 2^{2k}\qquad \hbox{ for every }
x \in S_n\ \hbox{ and }\, 0 \le k \le n.
\end{equation}
Note that the above inequality holds automatically for all $x \in S_n$ and
$0 \le k \le 1/(d-2)\cdot\log n$. A simple covering argument shows that it
also holds for all $x \in S_n$ and $0 \le k \le c_{17}\log n$, where $c_{17}$ is
the constant in \eqref{Eq:c17}. Hence, in order
to prove (\ref{Eq:claim}), it is sufficient to consider the event
\begin{equation}\label{Eq:claim2}
E_n = \bigg\{ \mu\big( Q_k(x)\big) > \gamma\, n\, 2^{2k}\ \hbox{ for some }
x \in S_n
\hbox{ and } \, c_{17} \log n \le k \le n\bigg\}
\end{equation}
and show that $\sum_{n=1}^\infty \P_0^\o(E_n) < \infty$. In the above, $\gamma>0$
is a generic constant whose value will be chosen later.

Since $\mu\big(Q_k(x)\big) > 0$ implies $Q_k(x)$ is hit by $\wh{X}$,
it follows from Lemma \ref{Lem:Hit} that for all
$c_{17} \log n \le k \le n-3$ (note that by \eqref{Eq:c17}, $2^{c_{17} \log n}
\geq n^{3d/2}$, hence Lemma \ref{Lem:Hit} applies), for every $x \in S_n $,
\[
   \P_0^\o \big(\wh{X}_n\in Q_k(x)\mbox{ for some } n\big)
   \leq    \P_0^\o \big(X_t\in Q_k(x) \mbox{ for some } t\big)
   \leq c 2^{(k-n)(d-2)}.
\]
By enlarging $c$ if necessary we can assume that the above inequality also
holds for  $k=n-2,$ $n-1$ and $n$.
Moreover, restarting at the hitting point (say, $\wh{X}_\tau$,
which necessarily lies in $V_{n+1}$)
and applying Lemma \ref{Lem:Sojourn}
and \eqref{Eq:M_F}, we have that for all $n$
large enough,
$$
\P_{\wh{X}_\tau}^\o \Big( \mu\big( Q_k(x)\big) > \gamma\, n\, 2^{2k}\Big)
\le \P_{\wh{X}_\tau}^\o \Big( T\big( Q_k(x)\big) > \gamma\, n\, 2^{2k}\Big)
\le e^{- \delta \gamma\, n }, \qquad \P^\o_0\hbox{-a.s.},
$$
where $\delta \in (0, 1)$ is a constant. Note that when applying
\eqref{Eq:M_F} we have again used the fact that $k \ge c_{17} \log n$.

It follows from the above and the strong Markov property that
\begin{equation}\label{Eq:claim3}
\begin{split}
\P_0^\o\big(E_n\big) &\le \sum_{ c_{17} \log n \le k \le n}
\P_0^\o \bigg\{\mu\big(Q_k(x)\big) > \gamma n\, 2^{2k}\ \hbox{ for some }
x \in S_n \bigg\}\\
&\le \sum_{ c_{17} \log n \le k \le n} c\, 2^{(n-k)d}\cdot 2^{(k-n)(d-2)}
\cdot e^{-\gamma \, \delta n}\\
&\le c\, e^{-(\gamma\, \delta -2) n},
\end{split}
\end{equation}
for all $n$ large enough. We now take $\gamma > 2/\delta$ so that
$\sum_{n=1}^\infty \P_0^\o(E_n) < \infty$. This and
the Borel-Cantelli lemma prove \eqref{Eq:claim}.

Hence, by \eqref{Eq:claim} and Lemma \ref{Lem:density}, we have
\begin{equation}\label{Eq:LB1}
\nu_2 \big( \wh{\rm R},  S_n \big)
\ge 2^{-d} \,c_{27}^{-1}\, n^{-1} 2^{-2n}\, \mu(S_n)
\end{equation}
for all $n$ large enough. By Lemma \ref{Lem:hit_pt} we have
\begin{equation}\label{Eq:mu_Sn_exp}
\E_0^\o\big( \mu(S_n)\big)
=\sum_{y\in S_n} \P_0^\o \big(\wh{X}_i = y\mbox{ for some } i \geq 0\big)
\ge c_{28}\, 2^{2n}
\end{equation}
for all $n$ large. This, (\ref{Eq:LB1}) and (\ref{Def:Hmeasure}) imply
$\E_0^\o\big(m_2( \wh{\rm R})\big) = \infty.$

Next we prove that $m_2\big( \wh{\rm R}\big) = \infty$\, $\P_0^\o$-a.s.
By (\ref{Eq:Smom}) and \eqref{Eq:M_F}, we have
\begin{equation}\label{Eq:LB2}
\E_0^\o\big( \mu(S_n)^2\big)
\le \E_0^\o\big( T(S_n)^2\big)
\le 2 \Big( M(S_n)\Big)^2
\le c_{29} 2^{4n}.
\end{equation}
Thus, by the Paley-Zygmund inequality (\cite[p.8]{Kahane85}), we obtain
\[
\begin{split}
\P_0^\o \Big( \mu(S_n) \ge \frac{1}{2} c_{28} 2^{2n}\Big)
&\ge \P_0^\o \Big( \mu(S_n) \ge  \frac 1 2 \E_0^\o\big(\mu(S_n)\big)\Big)\\
& \ge \frac 1 4 \frac{\Big(\E_0^\o\big(\mu(S_n)\big)\Big)^2}{\E_0^\o\big( \mu(S_n)^2\big)}\\
&\ge \frac{(c_{28})^2}{4 \,c_{29}} := p.
\end{split}
\]
Moreover, we can replace $\P_0^\o$ by $\P_x^\o$ and use Lemma \ref{Lem:S2} and the
same argument as above to show that the inequality
\begin{equation}\label{Eq:LB4b}
\P_x^\o \Big( \mu(S_n) \ge \frac 1 2 c_{28}\, 2^{2n}\Big) \ge p
\end{equation}
holds uniformly for all $n$ large and for all $x \in V(0, 2^{n-2})$.

We let $n_k = \lfloor \lambda k \log k\rfloor$, where $\lambda>0$ denotes
a constant whose value will be chosen later, and define
a sequence of stopping times by
\begin{equation}\label{Eq:LB5}
\tau_k = \inf\Big\{n> 0:\ \wh{X}_n \notin V\big(0, 2^{n_k}\big)\Big\}, \quad (k\ge 1).
\end{equation}
Note that $|\wh{X}_{\tau_{k+1}}| \ge 2^{n_{k+1}}/2$,
hence by using the strong Markov property and Lemma \ref{Lem:Hit} we obtain
that $\P_0^\o$ almost surely,
\begin{equation}\label{Eq:LB6}
\P_{\wh{X}_{\tau_{k-1}}}^\o\Big(\wh{X}_n \in S_{n_k} \hbox{ for some }\ n \ge \tau_{k+1}\Big)
\le c \bigg(\frac{2^{n_k}}{2^{n_{k+1}}}\bigg)^{d-2}\le \frac 1 {k^2},
\end{equation}
when the constant $\lambda$ is chosen large enough.

Next we consider
\begin{equation}\label{Eq:LB8}
\begin{split}
\P_0^\o\Big(|\wh X_{\tau_{k-1}}|> 2^{n_{k}-3}\Big)
&\le \P_0^\o\Big(|\wh X_{\tau_{k-1}}|> 2^{n_{k}-3},\ \tau_{k-1} \le 2^{2n_{k-1}}  n_{k-1}\Big)\\
&\qquad \qquad \qquad \quad + \P_0^\o\Big( \tau_{k-1} > 2^{2 n_{k-1}}  n_{k-1}\Big).
\end{split}
\end{equation}
Lemma \ref{Lem:Maximal} implies that
\begin{equation}\label{Eq:LB9}
\begin{split}
\P_0^\o\Big(|\wh X_{\tau_{k-1}}|> 2^{n_{k}-3},\ \tau_{k-1} \le 2^{2n_{k-1}}  n_{k-1}\Big)
&\le \P_0^\o\bigg(\sup_{0 \le t \le 2^{2n_{k-1}}  n_{k-1}} |X_{t}|> 2^{n_{k}-3}\bigg)\\
&\le c_{18}\exp\bigg(-c_{19}\, \frac{(k-1)^{2\log(2)\lambda}} { 64\lambda (k-1)\log(k-1)}\bigg) \\
&\le  c_{18}\exp(-c_{29} k)
\end{split}
\end{equation}
for all $k$ large enough when $\lambda$ is chosen large enough.  On the other hand,
by Lemmas~\ref{Lem:Sojourn}~and~\ref{Lem:S2} we have
\begin{equation}\label{Eq:LB10}
\begin{split}
\P_0^\o\Big( \tau_{k-1} > 2^{2 n_{k-1}}  n_{k-1}\Big)
\le \P_0^\o\Big(T\big(V(0, 2^{n_{k-1}})\big) \ge  2^{2 n_{k-1}}  n_{k-1}\Big)
\le c_{18}\exp(-c_{29} k)
\end{split}
\end{equation}
for all $k$ large enough, again when $\lambda$ is chosen large enough. Combining
(\ref{Eq:LB8}), (\ref{Eq:LB9}) and (\ref{Eq:LB10}) yields
\begin{equation}\label{Eq:LB11}
\P_0^\o\Big(| \wh X_{\tau_{k-1}}|> 2^{n_{k}-3}\Big) \le 2c_{18}\exp(-c_{29} k)
\end{equation}
for all $k$ large enough.

By (\ref{Eq:LB4b}) we have that $\P^\o_0$-a.s. on the event
$\big\{| \wh X_{\tau_{k-1}}|\le 2^{n_{k}-3}\big\}$,
\begin{equation}\label{Eq:LB12}
\P_{ \wh X_{\tau_{k-1}}}^\o \Big( \mu(S_{n_{k}})
\ge \frac 1 2 c_{28}\, 2^{2n_k}\Big) \ge p.
\end{equation}
Define
\[
\wh{\rm R}_{k}= \big\{x \in \Z^d: \wh{X}_n = x\ \hbox{ for some }\, \tau_{k-1} \le n <  \tau_{k+1}\big\}\
\]
to be the range of the discrete time VSRW $\wh{X}$ between the times $\tau_{k-1}$ and $\tau_{k+1}$.
Noting that $\wh X_\ell$ does not belong to $S_{n_k}$
when $\ell < \tau_{k-1}$, and using \eqref{Eq:LB6}, the strong Markov property and
\eqref{Eq:LB12}, one obtains that
for $k$ large,
on the event $\big\{| \wh X_{\tau_{k-1}}|\le 2^{n_{k}-3}\big\}$,
\begin{equation}\label{Eq:LB13}
\aligned
\P_{ \wh X_{\tau_{k-1}}}^\o \Big( \mu(\wh{\rm R}_{k} \cap S_{n_{k}}) \ge\frac 1 2 c_{28}\, 2^{2n_k}\Big)
&\ge  p   - \frac{1}{k^2}
\ge \frac p 2.
\endaligned
\end{equation}

Using \eqref{Eq:LB13} and \eqref{Eq:LB11} and applying Lemma \ref{Lem:SLLN} we conclude
that $\P_0^\o$-a.s. the inequality
\[
\mu(\wh{\rm R}_{2k} \cap S_{n_{2k}}) \ge \frac 1 2 c_{28}\, 2^{2n_k}
\]
holds for a sequence $K$ of integers $k$ of lower density at least
$\vep$ for some constant $\vep>0$.
 This
and \eqref{Eq:LB1} imply
\[
\nu_2 \big( \wh{\rm R}_{2k},  S_{n_{2k}} \big) \ge c_{30} n_{2k}^{-1},
\q\mbox{for all } k\in K.
\]
Therefore,
\[
m_2(\wh{\rm R}) \ge c_{30} \sum_{k \in K} \frac 1 {n_{2k}}
\ge c_{31} \sum_{k \in K} \frac 1 {2k \log (2k)}= \infty,
\qquad \hbox{ $\P_0^\o$-a.s.}
\]
This proves that $\dim \wh{\rm R} \ge 2$\, $\P_0^\o$-a.s. and the theorem.
\end{proof}

\begin{proof}{\it of Theorem \ref{Th:main2}}\, It follows from (\ref{Eq:Hittcap}),
(\ref{Eq:Hittcap2}) and (d) of Lemma \ref{Lem:BD} that
for all $n$ large enough,
\begin{equation}\label{Eq:HitCap}
\P^\o_0\Big( X_t \in A\cap S_n \ \hbox{ for some }\, t > 0\Big) \asymp \frac 1 {2^{n(d-2)}}
{\rm Cap}_{g^\o}(A\cap S_n).
\end{equation}
Hence the proof of Theorem in Lamperti \cite{Lam63} gives that
\[
\P^\o_0\Big( X_t \in A \ \hbox{ for  arbitrarily large}\, t > 0\Big) = 0
\]
if and only if $\sum_{n=1}^\infty {2^{-n(d-2)}} {\rm Cap}_{g^\o}(A\cap S_n)< \infty$.
The set $S_n$ here and the $S_n$ in \cite{Lam63} have different meanings, nevertheless
it is straightforward to modify the arguments in \cite{Lam63} to our setting.
The assumption (7) therein should be modified to: there exist
$a,b<\infty$ (depending on the environment)
such that for any $x\in S_n$, $y\in S_{n+m}$ where $n,m\geq b$,
\[
   g^\o(x,y)\leq a2^{-(n+m)(d-2)},\q g^\o(y,x)\leq a2^{-n(d-2)}.
\]
This holds thanks to (d) of Lemma \ref{Lem:BD} and \eqref{Eq:n0}.

Combining this with Lemma \ref{CapHau}, we deduce
\begin{itemize}
\item If $m_{h_1}(A) < \infty$, then $\P^\o_0\big( X_t \in A \
\hbox{ for  arbitrarily large}\, t > 0\big) = 0$.
\item If $m_{h_2}(A) = \infty$, then $\P^\o_0\big( X_t \in A \
\hbox{ for  arbitrarily large}\, t > 0\big) > 0$.
\end{itemize}
These
and Lemma \ref{01_law}
imply the conclusions of Theorem \ref{Th:main2}.
\end{proof}

\section{Proof of Theorem \ref{Th:BTM}}

The proofs of Theorems \ref{Th:main} and \ref{Th:main2} only make use
of Lemmas \ref{Lem:BD} and \ref{01_law}. Bounds on the Green's function
are used for estimating the hitting probabilities in Sections 3.1 and 3.2;
and the upper bound on the transition density $p_t^\o(x, y)$
is used to derive the maximal inequality in Section~3.3.

\begin{proof}{\it of Theorem \ref{Th:BTM}}\
If $a=0$, $\tilde X$ is a time change of the simple random walk on $\Z^d$.
Hence Theorem \ref{Th:BTM} follows from Theorems 7.8 and 8.3 in Barlow and Taylor
\cite{BT92}.

Assume $a \in (0, 1]$. It follows from Lemma 9.1 of  Barlow and  \u Cern\'y
\cite{BC11} and Theorem 6.1 in Barlow and Deuschel \cite{BD10} that,
under the assumption that $\tilde{\PP}(\kappa_x \ge 1) = 1$,
the transition density and the Green's function of
the VSRW $\tilde X$
satisfy Lemma \ref{Lem:BD}.
Moreover, by Lemma~9.1 and Proposition 3.2 in Barlow and  \u Cern\'y \cite{BC11},
$\tilde X$ also enjoys an elliptic Harnack inequality and hence a zero-one law as in
Lemma \ref{01_law} holds as well. Therefore, the proof of Theorem \ref{Th:BTM}
is the same as those of Theorems \ref{Th:main} and \ref{Th:main2}.
\end{proof}

\setcounter{section}{0}
\setcounter{subsection}{0}
\setcounter{equation}{0}
\renewcommand{\thesection}{\Alph{section}}
\renewcommand\thesubsection{\thesection.\arabic{subsection}}
\renewcommand{\theequation}{\thesection.\arabic{equation}}%

\section{Appendix: Proof of Lemma \ref{Lem:Cap}}

We first list some known facts about discrete time Markov chains with
discrete space $E$. Suppose $\{X_i,\, i\ge 0\}$ is such a process, for
any finite set $A$ in the state space, let
\begin{equation}\label{eq:hit_return}
  T_A=\inf\{i\geq 0: X_i \in A\},\q\mbox{and}\q S_A=\inf\{i> 0: X_i \in A\}
\end{equation}
be the first hitting time and the first return time of $A$ respectively. Then by
the last-exit decomposition, see, e.g.,
Proposition 3.5 in Revuz \cite[p.57]{Revuz84}
\begin{equation*}
\P_x(T_A<\infty) = \sum_{y\in A} g(x,y) \P_y(S_A = \infty),\q\mbox{ for all } x,
\end{equation*}
where $g(x,y)=\sum_{i=0}^\infty p_i(x,y)$ is the Green's function.
Moreover, if we define
\[
  {\rm Cap}(A)=\sum_{y\in A} P_y(S_A = \infty)
\]
to be the capacity of $A$, then it satisfies that Revuz \cite[Exercise 4.13 on p.64]{Revuz84}
\begin{equation}\label{eq:capacity_max_A}
{\rm Cap}(A) = \max\Big\{ \sigma(A): \, \sigma \hbox{ is a
measure on }\, A \hbox{ such that }\, \max_{x \in A}
(g\sigma)(x) \le 1 \Big\}.
\end{equation}

We now prove Lemma \ref{Lem:Cap}. For any $n$, define discrete time
Markov chain $\{X^{(n)}_i:=X_{i/2^n},\, i=0,1,\ldots\}$.
It has transition density $p^{(n)}_i(x,y)=p_{i/2^n}(x,y)$, and Green's function
$g^{(n)}(x,y)=\sum_{i=0}^\infty p_{i/2^n}(x,y).$
Now for any finite set $A$ in the state space $E$, define the hitting
time and return time $T_A^{(n)}$ and $S_A^{(n)}$ similarly as in
\eqref{eq:hit_return} for the process $X^{(n)}$. We then have that
\begin{equation}\label{eq:hit_prob}
\P_x\Big(T_A^{(n)}<\infty\Big) = \sum_{y\in A} g^{(n)}(x,y) \P_y\Big(S_A^{(n)}
= \infty\Big),\q\mbox{for all } x.
\end{equation}
Moreover, ${\rm Cap}^{(n)}(A)=\sum_{y\in A} \P_y(S_A^{(n)} = \infty)$ satisfies
\eqref{eq:capacity_max_A} with $g$ replaced by $g^{(n)}$.

We now let $n$ go to $\infty$. By the right continuity of the process $\{X_t, t \ge 0\}$,
\begin{equation}\label{eq:prob_hit_conv}
\P_x(T_A^{(n)}<\infty) \uparrow \P_x(T_A<\infty),\mbox{ where }
 T_A=\inf\{t\geq 0: X_t \in A\}.
\end{equation}
Moreover, by condition (ii) and Lemma 3.6.1 in Norris \cite{Norris98}, for any
$x, y \in E$ and any $0\leq s<t<\infty$,
\[
|p_s(x,y) -p_{t}(x,y)|\leq 1-e^{-q_x (t-s)}=O(t-s).
\]
Combining this with condition (i) one can verify that
\begin{equation}\label{eq:conv_Green}
   \frac{1}{2^n} g^{(n)}(x,y) \to g(x,y)=\int_0^\infty p_t(x,y)\,dt, \q\mbox{for all } x, y \in E.
\end{equation}
Furthermore, by condition (ii) again, for any $y\in A$ and $n \ge 1$,
\[
2^n\P_y(S_A^{(n)}=\infty)\leq 2^n \P_y(X_{1/2^n}\neq y) \leq q_y < \infty,
\]
hence $\{2^n\P_y(S_A^{(n)}=\infty):n\geq 1\}$ must admit a subsequence converging to some limit,
say $b(y).$
By \eqref{eq:hit_prob}, \eqref{eq:prob_hit_conv} and \eqref{eq:conv_Green} the $b(y)$'s
must satisfy
\[
\P_x(T_A<\infty)=\sum_{y\in A} g(x,y) b(y),\q\mbox{for all } x \in E.
\]
Thus we have explicitly built a function $b(y)$ which solves \eqref{Eq:Hittcap}.
Moreover, by the uniqueness of Riesz decomposition (see, for example, Syski
\cite[Theorem 1, p.165]{Syski92}), the solution to the above equation is unique,
and hence we conclude that the whole sequence $\{2^n\P_y(S_A^{(n)}=
\infty)\}$ must converge to $b(y)$.
We then have that
\[
{\rm Cap}(A)=\lim_{n\to \infty} 2^n {\rm Cap}^{(n)}(A).
\]
That it satisfies \eqref{eq:capacity_max} follows from the above
convergence and that ${\rm Cap}^{(n)}(A)$ satisfies \eqref{eq:capacity_max_A} .

Finally  the lemma applies to the VSRW, because condition (i) holds
thanks to (a) of Lemma~\ref{Lem:BD}, and (ii) holds by the definition of VSRW.
\qed

\section*{Acknowledgments}
We thank Martin Barlow for the proof of Lemma \ref{01_law}.
We also thank the editor and an anonymous referee for their helpful and
thoughtful comments and suggestions.

\bigskip

\bibliographystyle{abbrv}

\bigskip

\medskip
\textsc{Yimin Xiao:} Department of Statistics and Probability,
A-413 Wells Hall, Michigan State
University, East Lansing, MI 48824, U.S.A.\\
E-mail:  \texttt{xiao@stt.msu.edu }\\[1mm]

\medskip

\textsc{Xinghua Zheng:}   Department of ISOM,
Hong Kong University of Science and Technology, Clear Water Bay, Kowloon, Hong Kong\\
E-mail: \texttt{xhzheng@ust.hk}

\end{document}